\documentclass[12pt, reqno]{amsart}
\usepackage{amssymb,a4}
\usepackage{amsmath,amsthm,amscd}
\usepackage{amsthm}
\usepackage{amsfonts,amssymb}
\usepackage{mathrsfs}
\usepackage{amsmath}
\numberwithin{equation}{section}
\newtheorem{theo}{Theorem}[section]
\newtheorem{defi}[theo]{Definition}

\newtheorem{lemm}[theo]{Lemma}

\begin{document}

\title[2-local derivations on  infinite-dimensional Lie algebras]{2-Local derivations on the Super Virasoro algebra
 and Super ${\rm W(2,2)}$ algebra}

\author{ Munayim Dilxat}
\address{College of Mathematics and System Sciences, Xinjiang University, Urumqi, Xinjiang,830046, China}
\email{18396830969@163.com}

\author{Shoulan Gao}
\address{Department of Mathematics, Huzhou University, Huzhou, Zhejiang, 313000, China}
\email{gaoshoulan@zjhu.edu.cn}

\author{Dong Liu}
\address{Department of Mathematics, Huzhou University, Huzhou, Zhejiang, 313000, China}
\email{liudong@zjhu.edu.cn}

\date{}
\maketitle

 {\bf {\scriptsize Abstract:}}
 The present paper is devoted to study 2-local superderivations on the super Virasoro algebra and the super ${\rm W(2,2)}$ algebra. We prove that all 2-local superderivations on the super Virasoro algebra as well as the super ${\rm W(2,2)}$ algebra are  (global) superderivations.

 {\bf {\scriptsize Key Words:}} {\scriptsize Lie superalgebra, super ${\rm W(2,2)}$ algebra, 2-Local superderivation,  superderivation.}

{\bf {\scriptsize MSC:}} {\scriptsize  17B40, 17B65 }

   \section{Introduction }

In 1997, \v{S}emrl \cite{Sem} introduced the notion of  2-local derivations on algebras. Namely, a map \(\Delta : \mathcal{L} \to
\mathcal{L}\) (not necessarily linear) on an algebra \(\mathcal{L}\) is called a \textit{2-local
derivation} if, for every pair of elements \(x,y \in  \mathcal{L},\) there exists a
derivation \(D_{x,y} : \mathcal{L} \to \mathcal{L}\) such that
\(D_{x,y} (x) = \Delta(x)\) and \(D_{x,y}(y) = \Delta(y)\).  For a
given algebra \(\mathcal{L}\), the main problem concerning
these notions is to prove that they automatically become
derivations or to give examples of 2-local derivations of \(\mathcal{L},\)
which are not derivations. Solutions
of such problems for finite-dimensional Lie algebras over
algebraically closed field of zero characteristic were obtained in
\cite{AKR}. Namely, in
\cite{AKR} it is proved that every 2-local derivation on a
semi-simple Lie algebra is a derivation and that
each finite-dimensional nilpotent Lie algebra with dimension
larger than two admits 2-local derivation which is not a
derivation. Recently, there are some studies on 2-local derivations on the infinite-dimensional Lie algebras
\cite{AY, XT, ZCZ}.
The authors prove that 2-local derivations on the Witt algebra \cite{AY}, some class of generalized Witt algebra (or their Borel subalgebras) \cite{ZCZ} and W-algebra {\rm W(2,2)} \cite{XT} are derivations.

However, 2-local derivations for many Lie superalgebras are not studied up to now.
In the present paper,  we study 2-local superderivations on some
infinite-dimensional Lie superalgebras.

This paper is arranged as follows. In Section 2, we give some preliminaries concerning the super Virasoro algebra. In Section 3,
we prove that every 2-local superderivation on the super Virasoro  algebra is  automatically
a superderivation. In Section 4,
we prove that every 2-local superderivation on  the super ${\rm W(2,2)}$ algebra is also a superderivation.

Throughout this paper, we shall use $\mathbb C, \mathbb Z$ to denote the sets of the complex numbers and  the integers, respectively. All algebras (vector spaces) are based on the field $\mathbb C$.

 \section{Preliminaries}
In this section we give some necessary definitions and preliminary results.

\begin{defi}\cite{Kac,Sch}  Let $L$ be a Lie superalgebra, we call  a linear map $ D: L \rightarrow L $ a superderivation of $L$ if
\begin{equation*}
  D([x, y])= [D(x),y]+(-1)^{|D||x|} [x,D(y)], \forall x,y\in L.
\end{equation*}
\end{defi}
 Denote by ${\rm Der}(L)$ and ${\rm Der}_{\tau}(L)$ the set of all superderivations and all superderivations of degree $\tau$ of $L$ $(\tau\in\mathbb{Z}_{2})$, respectively. Obviously, ${\rm Der}(L)={\rm Der}_{\overline{0}}(L)\bigoplus {\rm Der}_{\overline{1}}(L)$.
 For all $a\in L$, the map ${\rm ad}(a)$ on $L$ defined as ${\rm ad}(a)x=[a,x], x \in L$ is a superderivation and superderivations of this form are called inner superderivations. Denote by ${\rm IDer}(L)$ the set of all inner superderivations of $L$.

 Recall that a map $\Delta: L\rightarrow L$ (not  linear in general)  is called a 2-local superderivation if, for every pair of elements $x,y\in L$, there exists a superderivation  $D_{x,y}: L\rightarrow L$ (depending on $x,y$)  such that $D_{x,y}(x)=\Delta(x)$ and $D_{x,y}(y)=\Delta(y)$.
  For a 2-local superderivation on $L$ and $k\in\mathbb{C}, x\in L$, we have
\begin{equation*}
\Delta(kx)= D_{x,kx}(kx)=kD_{x,kx}(x)=k\Delta(x).
\end{equation*}

Superconformal algebras have a long history in mathematical physics. The simplest examples, after the Virasoro algebra itself
(corresponding to $N = 0$) are the $N = 1$ superconformal algebras: the Neveu-Schwarz algebra ($\epsilon=\frac12$) and the
Ramond algebra ($\epsilon=0$). These infinite dimensional Lie superalgebras are also called the super-Virasoro algebras as they can be regarded as natural  super generalizations of the Virasoro algebra.

\begin{defi} For $\epsilon=0, \frac12$,  the super Virasoro algebra ${\rm SVir}[\epsilon]$ is a Lie superalgebra spanned by $ \{ L_{m}, G_{r}, C \mid m\in \mathbb{Z}, r\in\mathbb{Z}+\epsilon\}$, equipped with the following relations:
\begin{equation*}
[L_{m}, L_{n}]= (m-n)L_{m+n}+\frac{1}{12}\delta_{m+n, 0}(m^{3}-m)C,
\end{equation*}
\begin{equation*}
[L_{m}, G_{r}]= (\frac{m}{2}-r)G_{m+r},
\end{equation*}
\begin{equation*}
[G_{r}, G_{s}]= 2L_{r+s}+\frac{1}{3}\delta_{r+s, 0}(r^{2}-\frac{1}{4})C
\end{equation*}
for all $ m, n \in \mathbb{Z}, \ r, s\in \mathbb{Z}+\epsilon $.
\end{defi}
\begin{lemm}\cite{GMP}\label{lem2.4} For the super Virasoro algebra, we have $${\rm Der(SVir[\epsilon])}={\rm IDer(SVir[\epsilon])}.$$
\end{lemm}

The even part of the super Virasoro algebra is the Virasoro algebra. In \cite {AY} and \cite{ZCZ}, 2-local derivations over the Virasoro algebra are determined.

\begin{theo} \cite{AY, ZCZ}
Every 2-local derivation on  the Virasoro algebra  is a derivation.
\end{theo}

Based on the above methods, we do such researches for the super Virasoro algebra and some related Lie superalgebras in this paper.

\section{ 2-Local superderivations on  the super Virasoro algebra ${\rm SVir}[\epsilon]$}

Now we shall give the main result concerning 2-local superderivations on the super Virasoro  algebra ${\rm SVir}[0]$.

\begin{lemm} Let  $\Delta$ be a 2-local  superderivation  on the super Virasoro algebra ${\rm SVir}[0]$.  Then for
any $x,y\in {\rm SVir}[0] $, there exists a  superderivation   $D_{x,y}$  of  ${\rm SVir}[0]$ such that $D_{x,y}(x)=\Delta(x)$, $D_{x,y}(y)=\Delta(y)$ and it can be written as
\begin{equation}\label{3.1} D_{x,y}={\rm ad}(\sum_{k\in\mathbb{Z}}(a_{k}(x,y)L_{k}+b_{k}(x,y)G_{k})),\end{equation}
  where $a_{k},b_{k}(k\in\mathbb{Z})$ are complex-valued functions on ${\rm SVir}[0]\times {\rm SVir}[0]$.
\end{lemm}
\begin{proof}By Lemma  \ref{lem2.4},  the superderivation $D_{x,y}$ can obviously be written as the form of $(3.1)$.
\end{proof}

\begin{theo}\label{thm1} Every 2-local superderivation on  the super Virasoro algebra  ${\rm SVir}[0]$ is a superderivation.
\end{theo}

To prove Theorem 3.2, we need several lemmas.

\begin{lemm}Let $\Delta$ be a 2-local  superderivation on ${\rm SVir}[0]$.  For a given $i\in\mathbb{Z}$, if  $\Delta(G_{i})=0$ then
\begin{equation}\label{1}D_{G_{i},y}={\rm ad}(a_{2i}(G_{i},y)L_{2i}),  \qquad \forall  y\in {\rm SVir}[0],
\end{equation}
where $a_{2i}$ is a complex-valued function on ${\rm SVir}[0]\times {\rm SVir}[0]$.
\end{lemm}

\begin{proof} By Lemma 3.1, we can assume that
\begin{equation}\label{tt}
D_{G_{i},y}={\rm ad}(\sum_{k\in\mathbb{Z}}(a_{k}(G_{i},y)L_{k}+b_{k}(G_{i},y)G_{k}))
\end{equation}
 for some  complex-valued functions $a_{k},b_{k}(k\in\mathbb{Z})$  on ${\rm SVir}[0]\times {\rm SVir}[0]$.

When $\Delta(G_{i})=0$, in view of \eqref{tt} we obtain
\begin{eqnarray*}
0&=&\Delta(G_{i})
=D_{G_{i},y}(G_{i})
\\
&=&[\sum_{k\in\mathbb{Z}}(a_{k}(G_{i},y)L_{k}+b_{k}(G_{i},y)G_{k}),G_{i}]
\\
&=& \sum_{k\in\mathbb{Z}}((\frac{k}{2}-i)a_{k}(G_{i},y)G_{k+i}+2b_{k}(G_{i},y)L_{k+i}
+\frac{1}{3}\delta_{k+i, 0}(k^{2}-\frac{1}{4})b_{k}(G_{i},y)C).
\end{eqnarray*}
Then we have $(\frac{k}{2}-i)a_{k}(G_{i},y)=b_{k}(G_{i},y)=0$ for all $k\in\mathbb{Z}$, which deduces
$a_{k}(G_{i},y)=0$ with $k\neq2i$ and $b_{k}(G_{i},y)=0$ for all $ k\in\mathbb{Z}$. This with (3.3) implies that (3.2) holds. The proof is completed.
\end{proof}
\begin{lemm}  \label{lem2}  Let $\Delta$ be a $2$-local superderivation on ${\rm SVir}[0]$ such that $\Delta( G_{0})=\Delta( G_{1})=0$. Then
$$\Delta(G_{i})=\Delta(L_{i})=0,  \  \forall i\in\mathbb{Z}.$$
\end{lemm}
\begin{proof} Since $\Delta( G_{0})=\Delta( G_{1})=0$, by using Lemma 3.3, for any $y \in {\rm SVir}[0]$, we can assume that
\begin{equation}\label{01}D_{G_{0},y}={\rm ad}(a_{0}(G_0,y)L_{0}),\end{equation}
\begin{equation}\label{011}D_{G_{1},y}={\rm ad}(a_{2}(G_{1},y)L_{2}),\end{equation}
  where   $a_{0},a_{2}$  are complex-valued functions on ${\rm SVir}[0]\times {\rm SVir}[0]$.
Let $i\in\mathbb{Z}$ be a fixed index. Taking $y=G_{i}$ in \eqref{01} and \eqref{011} respectively,   we get
\begin{equation*}
\Delta(G_{i})=D_{G_{0},G_{i}}(G_{i})=[a_{0}(G_{0},G_{i})L_{0},G_{i}]=-ia_{0}(G_{0},G_{i})G_{i},
\end{equation*}
\begin{equation*}
\Delta(G_{i})=D_{G_{1},G_{i}}(G_{i})=[a_{2}(G_{1},G_{i})L_{2},G_{i}]=(1-i)a_{2}(G_{1},G_{i})G_{i+2}.
\end{equation*}
By the above two equations, we have
$$ ia_{0}(G_{0},G_{i})G_{i}+(1-i)a_{2}(G_{1},G_{i})G_{i+2}=0,$$
which implies $a_{0}(G_{0},G_{i})=0$ with $i\neq 0$ and $a_{2}(G_{1},G_{i})=0$ with $i\neq 1$.  It concludes that  $\Delta(G_{i})=0$.

Similarly, setting $y=L_{i}$ in \eqref{01} and \eqref{011} respectively,  we get
\begin{equation*}
\Delta(L_{i})=D_{G_{0},L_{i}}(L_{i})=[a_{0}(G_{0},L_{i})L_{0},L_{i}]=-ia_{0}(G_{0},L_{i})L_{i},
\end{equation*}
\begin{equation*}
\Delta(L_{i})=D_{G_{1},L_{i}}(L_{i})=[a_{2}(G_{1},L_{i})L_{2},L_{i}]
=a_2(G_{1},L_{i})((2-i)L_{i+2}+\frac12\delta_{i+2, 0}C).
\end{equation*}
By the above two equations, it follows that
$$ ia_{0}(G_{0},L_{i})L_{i}+a_2(G_{1},L_{i})((2-i)L_{i+2}+\frac12\delta_{i+2, 0}C)=0,$$
which implies $a_{0}(G_{0},L_{i})=0$ with $i\neq 0$ and $a_{2}(G_{1},L_{i})=0$ with $i\neq 2$.  It concludes that  $\Delta(L_{i})=0$.
 The proof is finished.
\end{proof}

\begin{lemm}\label{lem3}
Let $\Delta$ be a 2-local superderivation on ${\rm SVir}[0]$ such that $\Delta(G_{i})=0$ for all $i\in \mathbb{Z}$. Then for any $x=\sum_{t\in\mathbb{Z}}(\alpha_{t}L_{t}+\beta_{t}G_{t})+aC \in {\rm SVir}[0]$, where $\alpha_{t},\beta_t, a\in\mathbb{C}$, we have
\begin{equation*}
\Delta(x)=0.
\end{equation*}
\end{lemm}
\begin{proof}
For any $x=\sum_{t\in\mathbb{Z}}(\alpha_{t}L_{t}+\beta_{t}G_{t}) +aC\in {\rm SVir}[0]$, where $\alpha_{t},\beta_t, a\in\mathbb{C}$, since $\Delta(G_{i})=0$ for any $i\in\mathbb{Z}$, from Lemma 3.3 we have
\begin{eqnarray*}
\Delta(x)&=&D_{G_{i},x}(x)=[a_{2i}(G_{i},x)L_{2i},x]\\
  &=& \sum_{t\in\mathbb{Z}}(\alpha_{t}a_{2i}(G_{i},x)((2i-t)L_{t+2i}+\frac1{12}\delta_{2i+t, 0}(8i^3-2i)C)\\
  &&+(i-t)\beta_{t}a_{2i}(G_{i},x)G_{t+2i}).
\end{eqnarray*}
By taking enough diffident $i\in\mathbb{Z}$ in the above equation and, if necessary, let these $ i^{,} $s  be large enough, we obtain that $\Delta(x)=0$.
\end{proof}

Now we are to prove Theorem \ref{thm1}.

\textit{Proof of Theorem} \ref{thm1} : Let $\Delta$ be a 2-local superderivation on ${\rm SVir}[0]$. Take a derivation $D_{G_0,G_1}$ such that
\begin{equation*}
\Delta(G_0)=D_{G_0,G_1}(G_0), \ \Delta(G_1)=D_{G_0,G_1}(G_1).
\end{equation*}
Set $\Delta_1=\Delta-D_{G_0,G_1}$.  Then $\Delta_1$ is a 2-local
superderivation such that $\Delta_1(G_0)=\Delta_1(G_1)=0$.   By Lemma
\ref{lem2}, $\Delta_1(G_i)=0$ for all $i\in\mathbb{Z}$.  It
follows that $\Delta_1=0$.    Thus $\Delta=D_{G_0,G_1}$  is a
superderivation. The proof is completed.
\hfill$\Box$

Using the similar method of proving Theorem \ref{thm1}, we can obtain the same result for the  super Virasoro algebra  ${\rm SVir}[\frac{1}{2}]$, and then
we get the main result of this section.

\begin{theo} Every 2-local superderivation on  the super Virasoro algebra  ${\rm SVir}[\epsilon]$ is a superderivation.
\end{theo}

\section{ 2-local superderivation on the Super ${\rm W(2,2)}$ algebra}

In this section we shall study 2-local superderivations on the super ${\rm W(2,2)}$ algebra ${\rm SW(2,2)}$, which is an infinite dimensional Lie superalgebra introduced by
I. Mandal to  study the $N = (1, 1)$ supersymmetric
extension of Galilean conformal algebra (SGCA) in 2d (see \cite{GPB,M}).

By definition,  ${\rm SW(2,2)}$  is a Lie superalgebra over $\mathbb{C }$ with a basis
$$\{L_m, I_m, G_m, Q_{m},C_{1},C_{2}\mid m\in\mathbb{Z}\}$$
and the following non-vanishing relations:
\begin{eqnarray*}
&& [L_m,L_n]=(m-n)L_{n+m}+{1\over12}\delta_{m+n,0}(m^3-m)C_{1},\\
&&[L_m, I_n]=(m-n)I_{m+n}+{1\over12}\delta_{m+n,0}(m^3-m)C_{2},\\
&&[L_m,G_r]=(\frac{m}{2}-r)G_{m+r},\ \ \ \ \ \ \ \ [L_m,Q_r]=(\frac{m}{2}-r)Q_{m+r},\\
&&[G_r,G_s]=2L_{r+s}+{1\over3}\delta_{r+s,0}(r^{2}-\frac{1}{4})C_{1},\\
&&[G_r,Q_s]=2I_{r+s}+{1\over3}\delta_{r+s,0}(r^{2}-\frac{1}{4})C_{2}, \\
&&[I_m,G_r]=(\frac{m}{2}-r)Q_{m+r}
\end{eqnarray*}
for all  $m,n, r,s\in\mathbb{Z}$.

\begin{lemm}\cite{GPB}\label{lem4.1} For the Lie superalgebra ${\rm SW}(2,2)$,
$${\rm Der(SW(2,2))}={\rm IDer(SW(2,2))}\bigoplus\mathbb{C}D,$$
where $D$ is an outer derivation defined by
\begin{equation}\label{bb}
D(I_{m}) =I_{m}, D(Q_{r}) =Q_{r}, D(C_{2})=C_{2}, D(L_{m})=D(G_{r})=D(C_{1})=0
\end{equation}
for all $ m, r\in \mathbb{Z}$.
\end{lemm}

\begin{lemm} \label{lem4.2}Let  $\Delta$ be a 2-local  superderivation  on ${\rm SW(2,2)}$.  Then for any $x,y\in {\rm SW(2,2)}$, there exists a  superderivation   $D_{x,y}$ of ${\rm SW(2,2)}$ such that $D_{x,y}(x)=\Delta(x)$, $D_{x,y}(y)=\Delta(y)$ and it can be written as
\begin{equation}\label{aq}
\begin{split}
D_{x,y}&={\rm ad}\left(\sum_{k\in\mathbb{Z}}(a_{k}(x,y)L_{k}+b_{k}(x,y)I_{k}+c_{k}(x,y)G_{k}+d_{k}(x,y)Q_{k})\right)
+\lambda(x,y)D,
 \end{split}
 \end{equation}
where $a_{k},b_{k},c_{k},d_{k}$ and $\lambda$ are complex-valued functions on ${\rm SW(2,2)}\times {\rm SW(2,2)}$ and D is given by \eqref{bb}.
\end{lemm}
\begin{proof} By Lemma \ref{lem4.1},  the superderivation $D_{x,y}$ can obviously be written as the form of \eqref{aq}.
\end{proof}

Now we shall give the main result concerning 2-local superderivations on ${\rm SW(2,2)}$.

\begin{theo}\label{theo4.3} Every 2-local superderivation on  ${\rm SW(2,2)}$ is a superderivation.
\end{theo}

To prove Theorem 4.3, we need several lemmas.

\begin{lemm} \label{lem4.4} Let $\Delta$ be a 2-local  superderivation on ${\rm SW(2,2)}$.

$(i)$  For a given $ r\in\mathbb{Z}$, if $\Delta(G_{r})=0$, then for any  $y\in {\rm SW(2,2)}$,
\begin{equation}\label{04}
\begin{split}
D_{G_{r},y}&={\rm ad}(a_{2r}(G_{r},y)L_{2r}+b_{2r}(G_{r},y)I_{2r})+
\lambda(G_{r},y)D;
\end{split}
\end{equation}

$(ii)$ If $\Delta(I_{0}+Q_{0})=0$, then for any  $y\in {\rm SW(2,2)}$ we have
\begin{equation}\label{06}
\begin{split}
D_{I_{0}+Q_{0},y}&={\rm ad}(a_{0}(I_{0}+Q_{0},y)L_{0}+a_{1}(I_{0}+Q_0,y)L_{1}+\sum_{_{k\in\mathbb{Z}}}b_{k}(I_{0}+Q_0,y)I_{k}
\\&+
c_{1}(I_{0}+Q_0,y)G_{1}+\sum_{_{k\in\mathbb{Z}}}d_{k}(I_{0}+Q_0,y)Q_{k}),\end{split}
\end{equation}
where $a_{2r},b_{2r},a_0,a_{1},c_{1},b_{k},d_{k} (k\in\mathbb Z)$ and $\lambda$ are  complex-valued functions on ${\rm SW(2,2)}\times{\rm SW(2,2)}$.
\end{lemm}

\begin{proof} By Lemma \ref{lem4.2}, we can assume that

 \begin{equation}\label{e}
\begin{split}
D_{G_{r},y}&={\rm ad}(\sum_{k\in\mathbb{Z}}(a_{k}(G_{r},y)L_{k}+b_{k}(G_{r},y)I_{k}+
c_{k}(G_{r},y)G_{k}+d_{k}(G_{r},y)Q_{k}))\\&+
\lambda(G_{r},y)D,
\end{split}
\end{equation}

 \begin{equation}\label{t}
\begin{split}
D_{I_{0}+Q_{0},y}&={\rm ad}(\sum_{k\in\mathbb{Z}}(a_{k}(I_{0}+Q_{0},y)L_{k}
+b_{k}(I_{0}+Q_{0},y)I_{k}+
c_{k}(I_{0}+Q_{0},y)G_{k}\\&+ d_{k}(I_{0}+Q_{0},y)Q_{k}))+
\lambda(I_{0}+Q_{0},y)D,
\end{split}
\end{equation}
where $a_{k},b_{k},c_{k},d_{k}$ and $\lambda$ are  complex-valued functions on ${\rm SW(2,2)}\times{\rm SW(2,2)}$.

(i) When $\Delta(G_{r})=0$, in view of \eqref{e} we obtain
\begin{eqnarray*}
0&=&\Delta(G_{r})=D_{G_{r},y}(G_{r})
\\
&=&[\sum_{k\in\mathbb{Z}}(a_{k}(G_{r},y)L_{k}+b_{k}(G_{r},y)I_{k}+
c_{k}(G_{r},y)G_{k} +d_{k}(G_{r},y)Q_{k}), G_{r}]
\\
& &+\lambda(G_{r},y)D(G_{r})
\\
&=&\sum_{k\in\mathbb{Z}}((\frac{k}{2}-r)a_{k}(G_{r},y)G_{k+r}+(\frac{k}{2}-r)b_{k}(G_{r},y)Q_{k+r}
+ c_{k}(G_{r},y)(2L_{k+r}
\\
& &
+\frac13\delta_{k+r, 0}(k^2-\frac14)C_{1})+
  d_{k}(G_{r},y)(2I_{k+r}+\frac13\delta_{k+r, 0}(k^2-\frac14)C_{2}))
\end{eqnarray*}
 From the above equation, we have   $(\frac{k}{2}-r)a_{k}(G_{r},y)=(\frac{k}{2}-r)b_{k}(G_{r},y)=0$  for all $k\in\mathbb{Z}$, which deduces $ a_{k}(G_{r},y)=b_{k}(G_{r},y)=0$ with $k\neq 2r$. We also get
 $c_{k}(G_{r},y)=d_{k}(G_{r},y)=0$ for all $k\in\mathbb{Z}$. Then Equation \eqref{04} holds.

 (ii) When $\Delta(I_{0}+Q_{0})=0$, in view of \eqref{t} we obtain
\begin{eqnarray*}
0&=&\Delta(I_{0}+Q_{0})=D_{I_{0}+Q_{0},y}(I_{0}+Q_{0})
\\
&=&[  \sum_{k\in\mathbb{Z}}(a_{k}(I_{0}+Q_{0},y)L_{k}+b_{k}(I_{0}+Q_{0},y)I_{k}+
c_{k}(I_{0}+Q_{0},y)G_{k}
\\
& &+ d_{k}(I_{0}+Q_{0},y)Q_{k}),I_{0}+Q_{0}]+\lambda(I_{0}+Q_{0},y)D(I_{0}+Q_{0})
\\
&=&\sum_{k\in\mathbb{Z}}(ka_{k}(I_{0}+Q_{0},y)+2c_{k}(I_{0}+Q_{0},y))I_{k}
\\
& &+\sum_{k\in\mathbb{Z}}(\frac{k}{2}a_{k}(I_{0}+Q_{0},y)+kc_{k}(I_{0}+Q_{0},y))Q_{k}
\\
& &
  +\frac{1}{12}\delta_{k, 0}(k^3-k)a_{k}(I_{0}+Q_{0},y)C_{2}+\frac13\delta_{k, 0}(k^2-\frac14)c_{k}(I_{0}+Q_{0},y)C_{2}
\\
& &
  +\lambda(I_{0}+Q_{0},y)(I_{0}+Q_{0})
\end{eqnarray*}
If $k=0$, then $ 2c_{0}(I_{0}+Q_{0},y)I_{0}+\lambda(I_{0}+Q_{0},y)I_{0}+\lambda(I_{0}+Q_{0},y)Q_{0}=0$, we get $\lambda(I_{0}+Q_{0},y)=c_{0}(I_{0}+Q_{0},y)=0$ and $a_{0}(I_{0}+Q_{0},y)\neq 0$. If $k\neq 0$, then $ka_{k}(I_{0}+Q_{0},y)+2c_{k}(I_{0}+Q_{0},y)=0$ and $kc_{k}(I_{0}+Q_{0},y)+\frac{k}{2}a_{k}(I_{0}+Q_{0},y)=0$, we get $a_{k}(I_{0}+Q_{0},y)=c_{k}(I_{0}+Q_{0},y)=0$ with $k\neq 1$.
  Then \eqref{06} holds. The proof is completed.
\end{proof}

\begin{lemm}  \label{lem4.5}Let $\Delta$ be a 2-local superderivation on  ${\rm SW(2,2)}$ such that $\Delta( G_{0})=\Delta( G_{1})=0$. Then $$\Delta(G_{i})=\Delta(L_{i})=0,  \  \forall i\in\mathbb{Z}.$$
\end{lemm}

\begin{proof} It is essentially same as that of Lemma \ref{lem2}.
Since $\Delta( G_{0})=\Delta( G_{1})=0$, by using Lemma \ref{lem4.4}, for any $ y \in {\rm SW(2,2)}$,  we can assume that
\begin{equation}\label{a}
\begin{split}
D_{G_{0},y}&={\rm ad}(a_{0}(G_0,y)L_{0}+b_{0}(G_{0},y)I_{0})
+\lambda(G_{0},y)D,
\end{split}
\end{equation}
\begin{equation}\label{b}
\begin{split}
D_{G_{1},y}&={\rm ad}(a_{2}(G_1,y)L_{2}+b_{2}(G_{1},y)I_{2})+\lambda(G_{1},y)D,
\end{split}
\end{equation}
  where   $a_{0},a_2,b_0,b_{2},\lambda$ are  complex-valued functions on ${\rm SW(2,2)}\times{\rm SW(2,2)}$.

Let $i\in\mathbb{Z}$ be a fixed index. Taking $y=G_{i}$ in\eqref{a}  and \eqref{b} respectively, we get
$$\begin{array}{lll}
\Delta(G_{i})&=D_{G_{0},G_{i}}(G_{i})\\[2mm]
&=[a_{0}(G_{0},G_{i})L_{0}+b_{0}(G_{0},G_{i})I_{0}, G_{i}]
+\lambda(G_{0},G_{i})D(G_{i})\\[2mm]
  &= - ia_{0}(G_{0},G_{i})G_{i}- ib_{0}(G_{0},G_{i})Q_{i} \\[2mm]
 \end{array}$$ and
$$\begin{array}{lll}
\Delta(G_{i})&=D_{G_{1},G_{i}}(G_{i})\\[2mm]
&=[a_{2}(G_1,G_{i})L_{2}+b_{2}(G_{1},G_{i})I_{2}),G_{i}]+
\lambda(G_{1},G_{i})D(G_{i})\\[2mm]
  &=  (1-i)a_{2}(G_1,G_{i})G_{i+2}+(1-i)b_{2}(G_{1},G_{i})Q_{i+2}. \\[2mm]
 \end{array}$$
By the above two equations, it follows that
$$ ia_{0}(G_{0},G_{i})G_{i}+ ib_{0}(G_{0},G_{i})Q_{i}+ (1-i)a_{2}(G_1,G_{i})G_{i+2}+(1-i)b_{2}(G_{1},G_{i})Q_{i+2}=0,$$
which implies $a_{0}(G_{0},G_{i})=b_{0}(G_{0},G_{i})=0$ with $i\neq 0$ and $a_{2}(G_1,G_{i})=b_{2}(G_1,G_{i})=0$ with $i\neq 1$. It concludes that  $\Delta(G_{i})=0$.

Similarly, setting $y=L_{i}$ we can also get $\Delta(L_{i})=0$.
Taking $y=L_{i}$ in\eqref{a}  and \eqref{b} respectively,  we get
$$\begin{array}{lll}
\Delta(L_{i})&=D_{G_{0},L_{i}}(L_{i})\\[2mm]
&=[a_{0}(G_{0},L_{i})L_{0}+b_{0}(G_{0},L_{i})I_{0},L_{i}]
+\lambda(G_{0},L_{i})D(L_{i})\\[2mm]
  &= - ia_{0}(G_{0},L_{i})L_{i}- ib_{0}(G_{0},L_{i})I_{i} \\[2mm]
 \end{array}$$ and
$$\begin{array}{lll}
\Delta(L_{i})&=D_{G_{1},L_{i}}(L_{i})\\[2mm]
&=[a_{2}(G_1,L_{i})L_{2}+b_{2}(G_{1},L_{i})I_{2}),L_{i}]+
\lambda(G_{1},L_{i})D(L_{i})\\[2mm]
  &=  a_{2}(G_1,L_{i})((2-i)L_{i+2}+\frac12\delta_{i+2, 0}C_{1})+(2-i)b_{2}(G_{1},L_{i})I_{i+2}. \\[2mm]
 \end{array}$$
By the above two equations, it follows that
\begin{eqnarray*}
& &ia_{0}(G_{0},L_{i})L_{i}+ ib_{0}(G_{0},L_{i})I_{i}+a_{2}(G_1,L_{i})((2-i)L_{i+2}+\frac12\delta_{i+2, 0}C_{1})
\\
& &
+(2-i)b_{2}(G_{1},L_{i})I_{i+2}=0,
\end{eqnarray*}
which implies $a_{0}(G_{0},L_{i})=b_{0}(G_{0},L_{i})=0$ with $i\neq 0$ and $a_{2}(G_1,L_{i})=b_{2}(G_1,L_{i})=0$ with $i\neq 2$. It concludes that  $\Delta(L_{i})=0$.
The proof is finished.
\end{proof}

\begin{lemm}\label{lem4.6} Let $\Delta$ be a 2-local superderivation on  ${\rm SW(2,2)}$ such that $\Delta(G_{i})=0$ for all $i\in\mathbb{Z}$. Then for any $x=\sum_{t\in\mathbb{Z}}(\alpha_{t}L_{t}+\beta_{t}I_{t}+\gamma_{t}G_{t}
+\delta_{t}Q_{t})+aC_1+bC_2\in {\rm SW(2,2)}$, we have
\begin{equation*}
\begin{split}
\Delta(x)=
\mu_{x}\sum_{t\in\mathbb{Z}}(\beta_{t}I_{t}+\delta_{t}Q_{t})+bC_2,
\end{split}
\end{equation*}
where $\alpha_{t},\beta_t ,\gamma_{t},\delta_{t}, a, b,\mu_{x}\in \mathbb{C}$.
\end{lemm}
\begin{proof}
For $x=\sum_{t\in\mathbb{Z}}(\alpha_{t}L_{t}+\beta_{t}I_{t}+\gamma_{t}G_{t}
+\delta_{t}Q_{t})+aC_1+bC_2\in  {\rm SW(2,2)}$, since $\Delta(G_{i})=0$ for any $i\in\mathbb{Z}$, from Lemma \ref{lem4.4} we have
$$\begin{array}{lll}
\Delta(x)&=D_{G_{i},x}(x)=[a_{2i}(G_{i},x)L_{2i}+b_{2i}(G_{i},x)I_{2i},x]+\lambda(G_{i},x)D(x)\\[2mm]
  &=\sum_{t\in\mathbb{Z}}\alpha_{t}a_{2i}(G_{i},x)((2i-t)L_{2i+t}+\frac1{6}\delta_{2i+t, 0}(4i^3-i)C_1)\\[2mm]
   &+\sum_{t\in\mathbb{Z}}\alpha_{t}b_{2i}(G_{i},x)((2i-t)I_{2i+t}+\frac1{6}\delta_{2i+t, 0}(4i^3-i)C_2)\\[2mm]
  &+\sum_{t\in\mathbb{Z}}\beta_{t}a_{2i}(G_{i},x)((2i-t)I_{2i+t}+\frac1{6}\delta_{2i+t, 0}(4i^3-i)C_2)\\[2mm]
 &+\sum_{t\in\mathbb{Z}}((i-t)\gamma_{t}a_{2i}(G_{i},x)G_{2i+t}+
 ((i-t)\delta_{t}a_{2i}(G_{i},x)+\gamma_{t}b_{2i}(G_{i},x))Q_{2i+t})
 \\[2mm] & +\sum_{t\in\mathbb{Z}}\lambda(G_{i}, x)(\beta_{t}I_{t}+\delta_{t}Q_{t})+bC_2. \\[2mm]
 \end{array}$$
By taking enough diffident $i\in\mathbb{Z}$ in the above equation and, if necessary, let these $ i^{,} $s be large enough, we obtain that $\Delta(x)=\mu_{x}\sum_{t\in\mathbb{Z}}(\beta_{t}I_{t}+\delta_{t}Q_{t})+bC_2$. Furthermore, $\mu_{x}=\lambda( G_{i},x)$ is a constant since it is independent on $i$.
\end{proof}

\begin{lemm}\label{lem4.7} Let  $\Delta$ be a 2-local  superderivation  on ${\rm SW(2,2)}$ such that $\Delta(G_{i})=0$ and $\Delta(I_{0}+Q_{0})=0$  for all $i\in \mathbb{Z} $. Then for any odd integer $p$ and $y\in {\rm SW(2,2)}$,  there exist some $\zeta_{p}^{y},\eta_{p}^{y} \in \mathbb{C}$ such that
\begin{equation*}
\begin{split}
D_{L_{p}+I_{2p}+Q_{2p},y}&={\rm ad}(\zeta_{p}^{y}L_{p}+\eta_{p}^{y}I_{p}
+\zeta_{p}^{y}I_{2p}
+\zeta_{p}^{y}Q_{2p}).
\end{split}
\end{equation*}
\end{lemm}
\begin{proof} By $\Delta(G_{i})=0$ for all $i\in \mathbb{Z} $ and Lemma \ref{lem4.6}, we have
\begin{equation}\label{12}
\Delta(L_{p}+I_{2p}+Q_{2p})=
\mu_{L_{p}+I_{2p}+Q_{2p}}(I_{2p}+Q_{2p}),
\end{equation}
where $\mu_{L_{p}+I_{2p}+Q_{2p}}\in \mathbb{C}$. In view of $\Delta(I_{0}+Q_{0})=0$ and Lemma \ref{lem4.4}, we know that  $( x=L_{p}+I_{2p}+Q_{2p})$
$$\begin{array}{lll}
&\Delta(L_{p}+I_{2p}+Q_{2p})\\[2mm]
&=D_{I_{0}+Q_{0},L_{p}+I_{2p}+Q_{2p}}(L_{p}+I_{2p}+Q_{2p})\\[2mm]
&=[a_{0}(I_{0}+Q_0,x)L_{0}+a_{1}(I_{0}+Q_0,x)L_{1}+\sum_{k\in\mathbb{Z}}b_{k}(I_{0}+Q_0,x)I_{k}\\[2mm]
&+c_1(I_{0}+Q_0,x)G_{1}+\sum_{k\in\mathbb{Z}}d_{k}(I_{0}+Q_0,x)Q_{k}, L_{p}+I_{2p}+Q_{2p}]\\[2mm]
&=-pa_{0}(I_{0}+Q_0,x)(L_{p}+2I_{2p}+2Q_{2p})\\[2mm]
& +a_{1}(I_{0}+Q_0,x)((1-p)L_{p+1}+(1-2p)I_{2p+1}+(\frac{1}{2}-2p)Q_{2p+1})\\[2mm]
 &+\sum_{k\in\mathbb{Z}}(k-p)b_{k}(I_{0}+Q_0,x)I_{p+k}+c_{1}(I_{0}+Q_0,x)((1-\frac{p}{2})G_{p+1}
 \\[2mm]&+(1-p)Q_{2p+1}+2I_{2p+1})
 +\sum_{k\in\mathbb{Z}}(k-\frac{p}{2})d_{k}(I_{0}+Q_{0},x)Q_{k+p}. \\[2mm]
 \end{array}$$
Together with \eqref{12}, we get
$$\begin{array}{lll}
&\mu_{L_{p}+I_{2p}+Q_{2p}}(I_{2p}+Q_{2p})\\[2mm]
 &=-pa_{0}(I_{0}+Q_0,x)(L_{p}+2I_{2p}+2Q_{2p})\\[2mm]
 &+a_{1}(I_{0}+Q_0,x)((1-p)L_{p+1}+(1-2p)I_{2p+1}+(\frac{1}{2}-2p)Q_{2p+1})\\[2mm]
 &+\sum_{k\in\mathbb{Z}}(k-p)b_{k}(I_{0}+Q_0,x)I_{p+k}+c_{1}(I_{0}+Q_0,x)((1-\frac{p}{2})G_{p+1}
 \\[2mm]&+(1-p)Q_{2p+1}+2I_{2p+1})
 +\sum_{k\in\mathbb{Z}}(k-\frac{p}{2})d_{k}(I_{0}+Q_{0},x)Q_{k+p}. \\[2mm]
 \end{array}$$
 From the above equation we can easily see that $a_{0}(I_{0}+Q_{0},x)=0$.  Next, observe the coefficient of $I_{2p}$,  and then we  get $\mu_{L_{p}+I_{2p}+Q_{2p}}=0$. So it follows that
\begin{equation}\label{000}
\Delta(L_{p}+I_{2p}+Q_{2p})=0.
\end{equation}
Next, for any $y\in{\rm SW(2,2)}$, by Lemma \ref{lem4.2}, we can assume that
\begin{equation}\label{111}
\begin{split}
&D_{L_{p}+I_{2p}+Q_{2p},y}
\\
&
=ad(\sum_{k\in\mathbb{Z}}(a_{k}(L_{p}+I_{2p}
+Q_{2p},y)L_{k}+b_{k}(L_{p}+I_{2p}
\\
&+Q_{2p},y)I_{k}+c_{k}(L_{p}+I_{2p}+
Q_{2p},y)G_{k}+d_{k}(L_{p}+I_{2p}+Q_{2p},y)Q_{k}))
\\&+
\lambda(L_{p}+I_{2p}+Q_{2p},y)D.
 \end{split}
 \end{equation}
From  \eqref{000} and \eqref{111}, we have  $(x=L_{p}+I_{2p}+Q_{2p})$
$$\begin{array}{lll}
\Delta(x)&=D_{x,y}(x)\\[2mm]
&=[\sum_{k\in\mathbb{Z}}(a_{k}(x,y)L_{k}+b_{k}(x,y)I_{k}+c_{k}(x,y)G_{k}
+d_{k}(x,y)Q_{k}),x]\\[2mm]&+\lambda(x,y)D(x)\\[2mm]
&=\sum_{k\in\mathbb{Z}} a_{k}(x,y)((k-p)L_{p+k}+{1\over12}\delta_{p+k,0}(k^3-k)C_{1}+
(k-2p)I_{k+2p}\\[2mm]&
+{1\over12}\delta_{2p+k,0}(k^3-k)C_{2}+(\frac{k}{2}-2p)Q_{k+2p})
+\sum_{k\in\mathbb{Z}}b_{k}(x,y)((k-p)I_{p+k}\\[2mm]
 &+{1\over12}\delta_{p+k,0}(k^3-k)C_{2})
 +\sum_{k\in\mathbb{Z}} c_{k}(x,y)(2I_{2p+k}+{1\over3}\delta_{2p+k,0}(k^2-\frac{1}{4})C_{2}\\[2mm]&
 +(k-\frac{p}{2})G_{p+k}+(k-p)Q_{k+2p})+\sum_{k\in\mathbb{Z}}(k-\frac{p}{2})d_{k}(x,y)Q_{k+p}
 \\[2mm]&+\lambda(x,y)(I_{2p}+Q_{2p})=0. \\[2mm]
 \end{array}$$
From this, it is easy to see that $(k-p)a_{k}(x,y)L_{k+p}=0$  for all $k\in\mathbb{Z}$. So $a_{k}(x,y)=0$ for all $ k\neq p$. Similarly,  $(k-\frac{p}{2})c_{k}(x,y)G_{p+k}=0$ for all $k\in\mathbb{Z}$. Hence $c_{k}(x,y)=0$. Using this conclusion and observing the coefficients of $I_{3p}$, $Q_{3p}$ in the above equation, we get
\begin{equation*}
-pa_{p}(x,y)+(2p-p)b_{2p}(x,y)=0,
 \end{equation*}
\begin{equation*}
(\frac{p}{2}-2p)a_{p}(x,y)+(2p-\frac{p}{2})d_{2p}(x,y)=0,
 \end{equation*}
which implies $ a_{p}(x,y)=b_{2p}(x,y)=d_{2p}(x,y)$.
 Furthermore, by observing the coefficient of $I_{k},k \neq 3p$ in the above equation we get $\lambda(x,y)=0$ and $b_{k}(x,y)=0$ for all $k \neq p,2p$.
Observing the coefficient of $Q_{k},k \neq 3p$ in the above equation, we get $d_{k}(x,y)=0$ for all $k \neq 2p$.   Finally, set
$\zeta_{p}^{y}=a_{p}(x,y)$, $\eta_{p}^{y}=b_{p}(x,y)$, and then we finish the proof.
\end{proof}

\begin{lemm} \label{lem4.8}Let  $\Delta$ be a 2-local  superderivation  on ${\rm SW(2,2)}$ such that $\Delta(G_{0})=\Delta(G_{1})=\Delta(I_{0}+Q_{0})=0$. Then $\Delta(x)=0$ for all $x\in{\rm SW(2,2)}$.
\end{lemm}
\begin{proof} Take any but fixed $x=\sum_{t\in\mathbb{Z}}(\alpha_{t}L_{t}+\beta_{t}I_{t}+\gamma_{t}G_{t}
+\delta_{t}Q_{t})+aC_1+bC_2\in {\rm SW(2,2)}$, where $\alpha_{t}$, $\beta_{t}$, $\gamma_{t}$,
$\delta_{t}, a, b\in\mathbb C$ for any $t\in\mathbb Z$.

 Since $\Delta(G_{0})=\Delta(G_{1})=0$, it follows by Lemma \ref{lem4.5} that
\begin{equation}\label{g}
\Delta(G_{i})=0,  \forall i \in \mathbb{Z}.
 \end{equation}
This, together with Lemma \ref{lem4.6}, gives
\begin{equation}\label{x}
\begin{split}
\Delta(x)&=\Delta(\sum_{t\in\mathbb{Z}}(\alpha_{t}L_{t}+\beta_{t}I_{t}
+\gamma_{t}G_{t}
+\delta_{t}Q_{t})+aC_1+bC_2)
\\&=\mu_{x}\sum_{t\in\mathbb{Z}}(\beta_{t}I_{t}
+\delta_{t}Q_{t})+bC_2
\end{split}
 \end{equation}
for some $\mu_{x}\in\mathbb{C}$. By \eqref{g} and $\Delta(I_{0}+Q_{0})=0$, we obtain by Lemma \ref{lem4.7} that
\begin{equation}\label{y}
\begin{split}
\Delta(x)
&=D_{L_{p}+I_{2p}+Q_{2p},x}(x)
\\
&=[\zeta_{p}^{x}L_{p}+\eta_{p}^{x}I_{p}+\zeta_{p}^{x}I_{2p}+\zeta_{p}^{x}Q_{2p},
\\
&
\sum_{t\in\mathbb{Z}}(\alpha_{t}L_{t}
+\beta_{t}I_{t}+\gamma_{t}G_{t}+\delta_{t}Q_{t})+aC_1+bC_2]
\\
&=\sum_{t\in\mathbb{Z}}\zeta_{p}^{x}\alpha_{t}((p-t)L_{t+p}+\frac1{12}\delta_{p+t, 0}(p^3-p)C_1)
\\
&+\sum_{t\in\mathbb{Z}}\zeta_{p}^{x}\beta_{t}((p-t)I_{t+p}+\frac1{12}\delta_{p+t, 0}(p^3-p)C_2)
\\
&+\sum_{t\in\mathbb{Z}}\eta_{p}^{x}\alpha_{t}((p-t)I_{t+p}+\frac1{12}\delta_{p+t, 0}(p^3-p)C_2)
\\
&+\sum_{t\in\mathbb{Z}}\zeta_{p}^{x}\alpha_{t}((2p-t)I_{t+2p}+\frac1{12}\delta_{p+t, 0}(8p^3-2p)C_2)
\\
&+\sum_{t\in\mathbb{Z}}\zeta_{p}^{x}\gamma_{t}(2I_{t+2p}+\frac1{3}\delta_{2p+t, 0}(4p^2-\frac14)C_2)
\\
&+\sum_{t\in\mathbb{Z}}(\frac{p}{2}-t)\zeta_{p}^{x}\gamma_{t}G_{t+p}
\\
&+\sum_{t\in\mathbb Z}((\frac{p}{2}-t)\zeta_{p}^{x}\delta_{t}+(\frac{p}{2}-t)\eta_{p}^{x}\gamma_{t})Q_{t+p}
\\
&
+\sum_{t\in\mathbb{Z}}((p-t)\zeta_{p}^{x}\gamma_{t}+(2p-\frac{t}{2})\zeta_{p}^{x}\alpha_{t})Q_{t+2p}.
 \end{split}
 \end{equation}
Next the proof is divided into two cases according to the cases of $(\alpha_{t})_{t\in\mathbb{Z}}$ and $(\gamma_{t})_{t\in\mathbb{Z}}$.

{\bf Case 1.}
$(\alpha_{t})_{t\in\mathbb{Z}}$ is  not  a zero sequence, i.e., $x=\sum_{t\in\mathbb{Z}}(\alpha_{t}L_{t}+\beta_{t}I_{t}
+\gamma_{t}G_{t}+\delta_{t}Q_{t})+aC_1+bC_2$. Hence there is a nonzero term
$\alpha_{t_{0}}L_{t_{0}}$  in $x=\sum_{t\in\mathbb{Z}}(\alpha_{t}L_{t}+\beta_{t}I_{t}
+\gamma_{t}G_{t}+\delta_{t}Q_{t})+aC_1+bC_2$ for some $t_{0} \in \mathbb{Z}$. Take two integers $p=p_{1}$ and $p=p_{2}$ in \eqref{y} such that $p_{i}-t_{0}\neq 0,i=1,2$, then by
$(p_{i}-t_{0})\zeta_{p_{i}}^{x}\alpha_{t_{0}}L_{p_{i}+t_{0}}=0$ in \eqref{y} we have $\zeta_{p_{i}}^{x}=0$.  By
\eqref{x} and \eqref{y}, we have
$$\begin{array}{lll}
\Delta(x)&= \mu_{x}\sum_{t\in\mathbb{Z}}(\beta_{t}I_{t}
+\delta_{t}Q_{t})+bC_2\\[2mm]
&= \sum_{t\in\mathbb{Z}}\eta_{p_{i}}^{x}\alpha_{t}((p_{i}-t)I_{t+p_{i}}+\frac1{12}\delta_{p_{i}+t, 0}(p_{i}^3-p_{i})C_2)\\[2mm]&
+(\frac{p_{i}}{2}-t)\eta_{p_{i}}^{x}\gamma_{t}Q_{p_{i}+t}), i=1,2.\\[2mm]
 \end{array}$$
 Hence
$$ \mu_{x}\sum_{t\in\mathbb{Z}}\beta_{t}I_{t}= \sum_{t\in\mathbb{Z}}\eta_{p_{i}}^{x}\alpha_{t}((p_{i}-t)I_{t+p_{i}}+\frac1{12}\delta_{p_{i}+t, 0}(p_{i}^3-p_{i})C_2),$$
 $$\mu_{x}\sum_{t\in\mathbb{Z}}\delta_{t}Q_{t}= \sum_{t\in\mathbb{Z}}(\frac{p_{i}}{2}-t)\eta_{p_{i}}^{x}\gamma_{t}Q_{p_{i}+t}, i=1,2.$$
By taking $p_{1}$ and $p_{2}$  in the above equation such that $p_{1},p_{2},p_{1}-p_{2}$ are  large enough, we see that $\Delta(x)=0$.

{\bf Case 2.}
$(\alpha_{t})_{t\in\mathbb{Z}}$  is a zero sequence, i.e., $x=\sum_{t\in\mathbb{Z}}(\beta_{t}I_t+\gamma_{t}G_{t}+\delta_tQ_t)+aC_1+bC_2$.

 {\bf Subcase 2.1.} If $(\gamma_{t})_{t\in\mathbb{Z}}$  is  not a zero sequence,  there is a  nonzero term
$\gamma_{t_{0}}G_{t_{0}}$  in $x=\sum_{t\in\mathbb{Z}}(\beta_tI_t+\gamma_{t}G_{t}+\delta_tQ_t)+aC_1+bC_2$ for some $t_{0} \in \mathbb{Z}$. Take two integers $p=p_{1}$ and $p=p_{2}$ in \eqref{y} such that $\frac{p_{i}}{2}-t_{0}\neq 0,i=1,2$, then by
$(\frac{p_{i}}{2}-t_{0})\zeta_{p_{i}}^{x}\gamma_{t_{0}}G_{p_{i}+t_{0}}=0$ in \eqref{y} we have $\zeta_{p_{i}}^{x}=0$. By
\eqref{x} and \eqref{y}, we have
 $$\begin{array}{lll}
\Delta(x)= \mu_{x}\sum_{t\in\mathbb{Z}}(\beta_{t}I_{t}
+\delta_{t}Q_{t})+bC_2= \sum_{t\in\mathbb{Z}}(\frac{p_{i}}{2}-t)\eta_{p_{i}}^{x}\gamma_{t}Q_{p_{i}+t}, i=1,2.\\[2mm]
 \end{array}$$
 Hence $\Delta(x)=0$.

{\bf Subcase 2.2.} If $(\gamma_{t})_{t\in\mathbb{Z}}$  is a  zero sequence, i.e.,  $x=\sum_{t\in\mathbb{Z}}(\beta_{t}I_{t}
+\delta_{t}Q_{t})+aC_1+bC_2$.  Then by  \eqref{x} and \eqref{y} we have
$$\begin{array}{lll}
\Delta(x)&= \mu_{x}\sum_{t\in\mathbb{Z}}(\beta_{t}I_{t}
+\delta_{t}Q_{t})+bC_2\\[2mm]
&= \sum_{t\in\mathbb{Z}}( \zeta_{p}^{x}\beta_t((p-t)I_{t+p}+\frac1{12}\delta_{p+t, 0}(p^3-p)C_2)
+(\frac{p}{2}-t)\zeta_{p}^{x}\delta_{t}Q_{p+t}).\\[2mm]
 \end{array}$$
 By taking enough diffident $p$ in the above equation and, if necessary, let these $ p{'}$s  be large enough, we obtain that $\Delta(x)=0$.
 The proof is completed.
\end{proof}

Now we are to prove Theorem \ref{theo4.3}.

\textit{Proof of Theorem} \ref{theo4.3} : Let $\Delta$ be a 2-local superderivation on  ${\rm SW(2,2)}$. Take a superderivation $D_{G_0,G_1}$ such that
\begin{equation*}
\Delta(G_0)=D_{G_0,G_1}(G_0),\  \Delta(G_1)=D_{G_0,G_1}(G_1).
\end{equation*}
Set $\Delta_1=\Delta-D_{G_0,G_1}$.  Then $\Delta_1$ is a 2-local
superderivation such that $\Delta_1(G_0)=\Delta_1(G_1)=0$.   By Lemma
\ref{lem4.5}, $\Delta_1(G_i)=0$ for all $i\in\mathbb{Z}$. Combining with Lemma \ref{lem4.6}, we have $\Delta_{1}(I_{0}+Q_{0})=\mu_{I_{0}+Q_{0}}(I_{0}+Q_{0})$ for some
$ \mu_{I_{0}+Q_{0}}\in\mathbb{C}$. Now, set $\Delta_{2}=\Delta_{1}-\mu_{I_{0}+Q_{0}}D$, then $\Delta_{2}$ is a 2-local derivation such that
\begin{eqnarray*}
\Delta_{2}(G_{0})&=&\Delta_{1}(G_{0})-\mu_{I_{0}+Q_{0}}D(G_{0})=0,
\\
\Delta_{2}(G_{1})&=&\Delta_{1}(G_{1})-\mu_{I_{0}+Q_{0}}D(G_{1})=0,
\\
\Delta_{2}(I_{0}+Q_{0})&=&\Delta_{1}(I_{0}+Q_{0})
-\mu_{I_{0}+Q_{0}}D(I_{0}+Q_{0})
\\
&=&\mu_{I_{0}+Q_{0}}(I_{0}+Q_{0})-\mu_{I_{0}+Q_{0}}(I_{0}+Q_{0})=0.
\end{eqnarray*}
By Lemma \ref{lem4.8}, it follows that $\Delta_2=\Delta - D_{G_{0},G_{1}}-\mu_{I_{0}+Q_{0}}D\equiv0$.    Thus $\Delta=D_{G_0,G_1}+\mu_{I_{0}+Q_{0}}D$  is a
superderivation. The proof is completed.
\hfill$\Box$

\end{document}